\documentclass{amsart}
\usepackage[english]{babel}
\usepackage[letterpaper,top=2.5cm,bottom=2.5cm,left=2.5cm,right=2.5cm,marginparwidth=1.75cm]{geometry}
\usepackage{amsmath}
\usepackage{amsfonts}
\usepackage{amssymb}
\usepackage{amsthm}
\usepackage{mathtools}
\usepackage[colorlinks=true, allcolors=blue]{hyperref}
\usepackage{asymptote}
\usepackage{changepage}
\usepackage{graphicx}
\usepackage{url}
\usepackage[labelfont=bf]{caption}
\usepackage{array}

\newtheorem{theorem}{Theorem}
\newtheorem{lemma}[theorem]{Lemma}

\newtheorem{question}[theorem]{Question}
\newtheorem{proposition}[theorem]{Proposition}

\newtheorem*{remark}{Remark}
\newtheorem{faketheorem}[theorem]{Fake Theorem}

\newcommand{\RR}{\mathbb{R}}

\newcommand{\ZZ}{\mathbb{Z}}

\newcommand{\la}{\langle}
\newcommand{\rr}{\rangle}
\newcommand{\vol}{\operatorname{vol}}
\newcommand{\af}{\rightarrow}

\newcommand{\dd}{\mathrm{d}}
\DeclareMathOperator{\arccot}{\mathrm{arccot}}
\DeclareMathOperator{\Real}{\mathrm{Re}}

\DeclareMathOperator{\Res}{\mathrm{Res}}

\title{Simplex slicing: an asymptotically-sharp lower bound}
\author{Colin Tang}
\address{Department of Mathematical Sciences \\ Carnegie Mellon University \\ Wean Hall 6113 \\ Pittsburgh, Pennsylvania 15213 \\ United States of America}
\email{cstang@andrew.cmu.edu}

\date{\today}

\begin{document}



\begin{abstract}
We show that for the regular $n$-simplex, the $1$-codimensional central slice that's parallel to a facet will achieve the minimum area (up to a $1-o(1)$ factor) among all $1$-codimensional central slices, thus improving the previous best known lower bound (Brzezinski 2013) by a factor of $\frac{2\sqrt{3}}{e} \approx 1.27$. In addition to the standard technique of interpreting geometric problems as problems about probability distributions and standard Fourier-analytic techniques, we rely on a new idea, mainly \emph{changing the contour of integration} of a meromorphic function. 
\end{abstract}

\maketitle

\section{Introduction}\label{sec:real-intro}

We are broadly interested in the following question: if $K$ is a given convex body in $\RR^n$, and $H$ is some hyperplane of codimension $1$, can we obtain good upper and lower bounds on the area of the intersection $H \cap K$? (``Area'' meaning the $(n-1)$-dimensional volume.) We refer to $H \cap K$ as a \textbf{(hyperplane) section} of $K$. If $H$ is constrained to pass through the centroid of $K$, we call $H \cap K$ a \textbf{central section}.

%A major open problem which is related to this question is \emph{Bourgain's slicing problem}, which asks if $K$ is constrained to have volume $1$, can we always find a section of $K$ with area at least some absolute constant (independent of the dimension $n$)? This innocent-sounding question is related to isoperimetric inequalities in convex sets, and more generally to understanding uniform measures over high-dimensional objects. For a survey of Bourgain's slicing problem, consult \cite{klartag-survey}.

%Another motivation for studying the area of $H \cap K$ comes from the \emph{Busemann-Petty problem}, which asks that if two symmetric bodies $K_1, K_2$ satisfy $\vol_{n-1}(H \cap K_1) \le \vol_{n-1}(H \cap K_2)$ for each $H$ passing through the center, must it follow that $\vol_n(K_1) \le \vol_n(K_2)$? In other words, if every central section of $K_1$ is at most the corresponding central section of $K_2$, must $K_1$ have smaller volume than $K_2$?

%Returning to the question in the first paragraph,
Let us consider the case where $K$ is the hypercube $Q_n \coloneqq [-\frac12, \frac12]^n$ and $H$ is constrained to pass through the origin (which is the centroid of $Q_n$). The first sharp result was discovered independently by Hadwiger (\cite{hadwiger}) and Hensley (\cite{hensley}), who showed that the area of $H \cap Q_n$ is minimized by taking $H$ perpendicular to the vector $(1,0,0,0,\dots,0)$. (In this case $H \cap Q_n$ is isometric to $Q_{n-1}$.) Hensley's proof relied on a technique which has become commonplace: \textbf{\textit{converting the original geometric problem into the language of probability distributions}}. Later in the same paper, Hensley used another technique which has become commonplace, namely the use of the \textbf{\textit{Fourier inversion formula}} to deal with tricky convolutions.\footnote{See \cite{fourier} for more exposition on the method of Fourier analysis in convex geometry.} Ball used both aforementioned techniques in \cite{ball-cube} to prove that the area of $H \cap Q_n$ is maximized by taking $H$ perpendicular to the vector $(1,1,0,0,0,0,\dots,0)$. (In this case $H \cap Q_n$ is a rectangular prism with one side $\sqrt{2}$ and all other sides of length $1$.) Incidentally, Ball's result immediately implies that for $n \ge 10$, any central sections of $Q_n$ will have area less than that of any central section of $B(\Gamma(\frac{n}{2}+1)^{1/n}/\sqrt{\pi})$, the $n$-ball of volume $1$ (see \cite{ball-busemann-petty}). This gives a negative answer to the Busemann-Petty problem for all $n \ge 10$.

Now let us turn our attention to the case where $K$ is the regular simplex $\Delta_n$ (of side length $\sqrt{2}$). If $H$ is not constrained to pass through the center of $\Delta_n$, Webb noted in the introduction of \cite{webb} that some results of Ball (involving maximum-volume inscribed ellipsoids) can be used to show that the area of $H \cap \Delta_n$ is maximized by letting $H$ be a facet of $\Delta_n$. Fradelizi (in \cite{fradelizi}) has given a different proof of this same result. However, if $H$ is constrained to pass through the center of $\Delta_n$, Webb showed in the same paper (\cite{webb}) that the area of $H \cap \Delta_n$ is maximized by taking $H$ to pass through all but two vertices of $\Delta_n$ (a quantitative version of Webb's result is given in \cite{myroshnychenko}). Webb gave two proofs of this result, one of which used both the technique of probability distributions and the technique of the Fourier inversion formula. This leaves open the following question, which is the focus of this paper:

\begin{question}\label{q:main}
    If $H$ is constrained to pass through the center of $\Delta_n$, then what is the minimum possible value of the area of $H \cap \Delta_n$? In other words, what is the minimum central section of the regular simplex? Furthermore, which choices of $H$ attain this minimum value?
\end{question}

It has been conjectured that the minimum in Question \ref{q:main} is attained by taking $H$ to be $H_{\text{facet}}$, where $H_\text{facet}$ is any hyperplane through the center of $\Delta_n$ and parallel to a facet (this is Conjecture 4 in \cite{tkocz-survey}). Filliman states this conjecture without proof in Section 3, part (h) of \cite{filliman}. Dirksen showed (Theorem 1.3 (i) in \cite{dirksen}) that $H_\text{facet}$ is a local minimizer; i.e. for all $H$ sufficiently close to $H_\text{facet}$, the area of $H \cap \Delta_n$ is at least that of $H_\text{facet} \cap \Delta_n$. In Theorem 1.1 of \cite{brzezinski}, Brzezinski showed a lower bound that is within a factor of $1.27$ of the conjectured minimizer; i.e. Brzezinski proved that for any $H$, the area of $H \cap \Delta_n$ is at least $\frac{e}{2\sqrt{3}} \approx \frac{1}{1.27}$ times that of $H_\text{facet} \cap \Delta_n$. One reason why the question of minimizing $H \cap \Delta_n$ is harder than the analogous question for $Q_n$ is that Hensley was able to utilize the fact that $Q_n$ is a symmetric convex body, so the probability distribution resulting from slicing $Q_n$ is a symmetric probability distribution. The simplex is not a symmetric convex body, so the resulting probability distribution is in general not symmetric.

In this paper we almost resolve the question of minimizing $H \cap \Delta_n$; i.e. we prove a lower bound that is within a factor of $1 - o(1)$ of the conjectured minimizer (the notation $o(1)$ is with respect to taking $n \nearrow \infty$).

\begin{theorem}\label{thm:main}
    For each hyperplane $H$ passing through the center of $\Delta_n$, we have \begin{equation*}\begin{split}\vol_{n-1}(H \cap \Delta_n) &\ge \frac{1}{e}\sqrt{\frac{n+1}{n}}\left(\frac{n+1}{n}\right)^{n-1} \vol_{n-1}(H_{\mathrm{facet}} \cap \Delta_n) \\ &= \frac{1}{e}\frac{\sqrt{n+1}}{(n-1)!}.\end{split}\end{equation*}
\end{theorem}

Note that $\frac{1}{e}\sqrt{\frac{n+1}{n}}\left(\frac{n+1}{n}\right)^{n-1} \nearrow 1$ as $n \nearrow \infty$, so Theorem \ref{thm:main} is best possible (asymptotically speaking).

Our main new idea is to \textbf{\textit{move the contour of integration}}. Following the aforementioned techniques of probability distributions and of the Fourier inversion formula, we rewrite Theorem \ref{thm:main} as a statement that the probability density of a sum of independent exponential random variables takes a value of at least $\frac{1}{e}$ at zero. Specifically, we show the following, which may have independent interest:

\begin{theorem}\label{thm:prob-distro}
    Let $Y_1,Y_2,\dots,Y_{n+1}$ be i.i.d. standard exponential random variables (mean $1$). For each unit vector $u \in \RR^{n+1}$, we have that the density at $0$ of the random variable $Z \coloneqq u_1(Y_1-1) + u_2(Y_2-1)+\dots+u_{n+1}(Y_{n+1}-1)$ is at least $\frac{1}{e}$. Equality holds if $u$ is of form $(1,0,0,0,\dots,0)$.
\end{theorem}

Using Fourier inversion, we rewrite Theorem \ref{thm:prob-distro} as a statement that the integral of the characteristic function of $Z$ is at least $\frac{1}{e}$. However, $F$ is highly oscillatory along the real axis, and is difficult to bound. We simply move the contour of integration to some new contour $\gamma$, so that $F$ takes only positive real values along $\gamma$; now we can bound the integral more easily.

Our proof appears to be highly dependent on the idiosyncracies of the regular simplex, and do not generalize well to other convex bodies. In particular, there is a curious inequality (Equations (\ref{eq:cauchy-1}) and (\ref{eq:cauchy-2})), and it is not clear if there is an analogue when the simplex is replaced with some other convex body.

A failed approach that is worth mentioning is the following false statement, strengthening Theorem \ref{thm:prob-distro}:

\begin{faketheorem}\label{fakethm}
Let $f_X$ be the density function of a log-concave random variable $X$, such that $X$ has mean $0$ and variance $1$. Then $f_X(0) \ge \frac{1}{e}$, with equality holding if $X$ is equal to $Y-1$, where $Y$ is a standard exponential random variable (mean $1$).
\end{faketheorem}

This is indeed a strengthening of Theorem \ref{thm:prob-distro}, since the random variable $Z$ from Theorem \ref{thm:prob-distro} is the sum of independent log-concave random variables, and is hence log-concave. However, the Fake Theorem \ref{fakethm} is indeed false, as can be seen by taking $X$ to be the uniform distribution with mean $0$ and variance $1$; we end up with $f_X(0) = \frac{1}{\sqrt{12}} < \frac{1}{e}$. (In fact, $\frac{1}{\sqrt{12}}$ is precisely the minimum possible value of $f_X(0)$, as $X$ varies over all log-concave random variables with mean $0$ and variance $1$.)

%A good survey of questions about slicing convex bodies is in Section 2 of \cite{tkocz-survey}.

Finally, we will note in passing that, letting $K$ be a convex body and $H$ be a hyperplane as before, people have considered constraining $H$, not so that it passes through the centroid of $K$, but so that it is at a distance exactly $t$ from the centroid of $K$, where $t>0$ is a fixed constant. In this case $H \cap K$ is called a \textbf{noncentral section}. K\"{o}nig showed (Theorem 1.1 of \cite{konig}) that if $t$ is reasonably large, the area of $H \cap \Delta_n$ is maximized by taking $H$ to be parallel to a facet of $\Delta_n$. Liu and Tkocz investigated the cross-polytope (\cite{tkocz-cross}), and Moody, Stone, Zach, and Zvavitch investigated the hypercube (\cite{moody}).

%For specific choices of $K$ (such as the hypercube or the regular simplex), we can associate to each $H$ a probability distribution $G_H$ on the real line, so that the area of $H \cap K$ corresponds to the density at zero of $G_H$. The question of obtaining bounds on $H \cap K$ is thus related to the densities of sums of independent random variables.

In the remainder of this section we will sketch the proof of Theorem \ref{thm:main}. 
\begin{itemize}
\item In Section \ref{sec:prob}, we first use a result of Webb (Equation (\ref{eq:prob-interpret})) to reduce Theorem \ref{thm:main} to a question about the distribution of a sum of independent centered exponential random variables (Theorem \ref{thm:minimizer-prob-strong}, which is the same as Theorem \ref{thm:prob-distro}). This constitutes a use of the technique of probability distributions.

\item In Section \ref{sec:fourier}, we take the Fourier transform, and the problem reduces to showing that the integral of an explicit meromorphic function $F_a$ along the real line is at least $\frac{1}{e}$ (Theorem \ref{thm:1-e}). Here, $F_a$ is the characteristic function of the probability distribution. This constitutes a use of the technique of the Fourier inversion formula.

\item In Section \ref{sec:change-contour}, we change the contour of integration; this is our main new idea. As noted previously, $F_a$ is too difficult to bound along the real axis, so we instead change the contour of integration from the real axis to some path $\gamma_a$, with the property that $F_{a}$ always attains positive real values on $\gamma_a$. Since the integral of a meromorphic function does not depend on the contour used\footnote{As long as one can move one contour to the other contour without passing over any poles.}, the problem reduces to showing that the integral of $F_a$ along $\gamma_{a}$ is at least $\frac{1}{e}$ (Theorem \ref{thm:1-e-2}). 

\item In Section \ref{sec:pointwise}, we give a lower bound for the values of $F_{a}$ along $\gamma_{a}$. Essentially, we let $v$ denote the vector $(1) \in \RR^1$, and we show that $F_{a}(x) \ge F_v(x)$ for each $x$ (Theorem \ref{thm:pointwise}). We do this by showing that the derivative of $F_{a}$ is at least the derivative of $F_{v}$ (Equation (\ref{eq:log-ineq})). There is a curious inequality involved (Equations (\ref{eq:cauchy-1}) and (\ref{eq:cauchy-2})). We obtain the desired $\frac{1}{e}$ lower-bound by evaluating the integral of $F_{v}$.
\end{itemize}

\section{Definitions}\label{sec:def}

Let $n \ge 2$ be a positive integer. Define the standard basis vectors $e_1,e_2,\dots,e_{n+1}$ in $\RR^{n+1}$, such that $e_j$ denotes the vector with a $1$ in the $j$th coordinate, and zeroes in all other entries. We use $\Delta_n$ to denote the regular $n$-simplex of side length $\sqrt{2}$, embedded in $\RR^{n+1}$ as follows: $\Delta_n$ is the convex hull of the points $e_1,e_2,\dots,e_{n+1}$. Equivalently, \[\Delta_n = \left\{(x_1,x_2,x_3,\dots,x_{n+1}) \in \RR^{n+1} \bigm\vert  x_1+x_2+x_3+\dots+x_{n+1} = 1, \text{ and } x_i \ge 0 \text{ for each } i\right\}.\]

For any set $Q$, we let $\vol_k(Q)$ denote the $k$-dimensional Hausdorff measure of $Q$. Then the volume of $\Delta_n$ is given by \[\vol_n(\Delta_n) = \frac{\sqrt{n+1}}{n!}.\] 

Let $\vec{\mathbf{1}} \in \RR^{n+1}$ denote the unit vector $\frac{1}{\sqrt{n+1}}(1,1,1,\dots,1) = \left(\frac{1}{\sqrt{n+1}},\frac{1}{\sqrt{n+1}},\frac{1}{\sqrt{n+1}},\dots,\frac{1}{\sqrt{n+1}}\right)$. %and let $\hat{\mathbf{1}} \in \RR^{n+1}$ denote the corresponding unit vector; i.e., \[\hat{\mathbf{1}} = \frac{1}{\sqrt{n+1}}\vec{\mathbf{1}} = \left(\frac{1}{\sqrt{n+1}},\frac{1}{\sqrt{n+1}},\frac{1}{\sqrt{n+1}},\dots,\frac{1}{\sqrt{n+1}}\right).\] 
Observe that $\frac{1}{\sqrt{n+1}}\vec{\mathbf{1}}$ is the center of the regular simplex $\Delta_n$. Furthermore, given any two vectors $v, w \in \RR^{n+1}$, we write $\la v,w \rr$ for their dot product, and we say that $v \perp w$ if $\la v, w \rr = 0$, i.e. if $v, w$ are orthogonal. 

Throughout this paper, let $a \in \RR^{n+1}$ denote an arbitrary unit vector satisfying $a \perp \vec{\mathbf{1}}$. Let $H(a)$ denote the hyperplane perpendicular to $a$; i.e. \[H(a) = \left\{x \in \RR^{n+1} \bigm\vert x \perp a\right\}.\] Note that since $a \perp \vec{\mathbf{1}}$, we have that $H(a)$ contains the point $\frac{1}{\sqrt{n+1}}\vec{\mathbf{1}}$; i.e. $H(a)$ always passes through the center of $\Delta_n$.

The intersection $H(a) \cap \Delta_n$ is an $(n-1)$-dimensional polytope embedded in $\RR^{n+1}$. We call this object the \textbf{central section of $\Delta_n$ in the direction of $a$}, and denote it by $T(a)$. As in Question \ref{q:main}, we will be concerned with finding the minimum possible value of $\vol_{n-1}(T(a))$ as $a$ varies over all unit vectors perpendicular to $\vec{\mathbf{1}}$.

    \begin{figure}
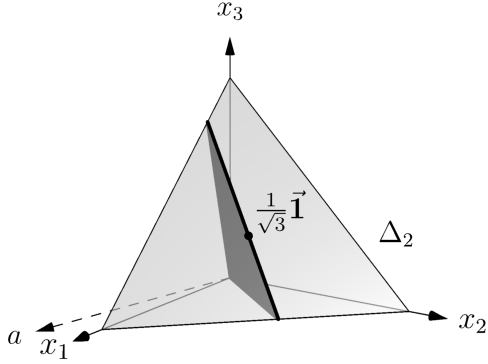

    \centering
\begin{asy}
settings.outformat = "pdf";
settings.prc = false;
//settings.render = 8;
import three;
import graph;

unitsize(3cm);
currentprojection = perspective(6,4,2);
real ax = 1.2;
real sc = 1.1;
draw((0,0,0)--(ax,0,0),arrow=Arrow3());
label("$x_1$",sc*(ax,0,0));
draw((0,0,0)--(0,ax,0),arrow=Arrow3());
label("$x_2$",sc*(0,ax,0));
draw((0,0,0)--(0,0,ax),arrow=Arrow3());
label("$x_3$",sc*(0,0,ax));
path3 outline = (1,0,0)--(0,1,0)--(0,0,1)--cycle;
draw(outline);
draw(surface(outline),gray(0.8)+opacity(0.6));
label("$\Delta_2$",sc*(-0.09,0.79,0.3));
real p1 = -10;
real p2 = -3;
real p3 = 13;

triple a = (p3,p1,p2)/sqrt(p1*p1+p2*p2+p3*p3);
draw((0,0,0)--a,arrow=Arrow3(),dashed);
label("$a$",sc*a);
triple a1 = (p2,0,-p3)/(p2-p3);
triple a2 = (p1,-p3,0)/(p1-p3);
path3 plane = (0,0,0)--a1--a2--cycle;
draw(plane,linewidth(1.5));
draw(surface(plane),black);

dot((1/3,1/3,1/3),linewidth(3.5));
label("$\frac{1}{\sqrt{3}} \vec{\mathbf{1}}$", sc*(0.2,0.4,0.4));
\end{asy}
        \caption{The gray triangle represents $\Delta_2$, and the black triangle represents $H(a)$. The thick line at the intersection of the two triangles represents the hyperplane section $T(a)$.}
    \end{figure}

For convenience, with $a$ defined as above, we will write the entries of $a$ as \begin{equation}\label{eq:a-coords}a = (a_1, a_2, \dots, a_{n+1}).\end{equation} Observe that \[a_1+a_2+\dots+a_{n+1} = 0\] and \[a_1^2+a_2^2+\dots+a_{n+1}^2=1.\] 

%For any measurable function $Q:\RR\af\RR$, we write $\|Q\|_1, \|Q\|_\infty$ to denote respectively the $L^1$, $L^\infty$ norm of $Q$.

%\section{Outline}

%Throughout this paper, we will be concerned with the following question:
%\begin{question}\label{q}
%Let $n$ be fixed. What is the smallest possible value that $\vol_{n-1}(T(\hat{a}))$ can achieve, as $\hat{a}$ ranges over all unit vectors in $\RR^{n+1}$ which are perpendicular to $\hat{\mathbf{1}}$?
%\end{question}
%Informally, this question asks for the smallest possible volume of the intersection of $\Delta^n$ with a hyperplane passing through its center.

Define the unit vector \begin{equation*}\begin{split}a_{\text{facet}} &\coloneqq \frac{n}{\sqrt{n(n+1)}}e_1-\frac{1}{\sqrt{n(n+1)}}(e_2+e_3+\dots+e_{n+1}) \\ &= \left(\frac{n}{\sqrt{n(n+1)}},-\frac{1}{\sqrt{n(n+1)}},-\frac{1}{\sqrt{n(n+1)}},-\frac{1}{\sqrt{n(n+1)}},\dots,-\frac{1}{\sqrt{n(n+1)}}\right).\end{split}\end{equation*} Observe that $H(a_{\text{facet}})$ corresponds exactly to the hyperplane $H_{\text{facet}}$ from Section \ref{sec:real-intro}, which corresponds to the conjectured minimum central section (as in Question \ref{q:main}). We compute that $T(a_{\mathrm{facet}})$ is the convex hull of the $n$ vertices \[\frac{1}{n+1}e_1 + \frac{n}{n+1}e_2, \frac{1}{n+1}e_1 + \frac{n}{n+1}e_3, \frac{1}{n+1}e_1 + \frac{n}{n+1}e_4, \dots, \frac{1}{n+1}e_1 + \frac{n}{n+1}e_{n+1}\] and thus \[\vol_{n-1}(T(a_\text{facet})) = \frac{\sqrt{n}}{(n-1)!}\left(\frac{n}{n+1}\right)^{n-1},\] verifying the equality in the second half of Theorem \ref{thm:main}. 
Thus Theorem \ref{thm:main} can be rewritten as follows:

\begin{theorem}\label{thm:minimizer}
    For each unit vector $a \in \RR^{n+1}$ satisfying $a \perp \vec{\mathbf{1}}$, we have $\vol_{n-1}(T(a)) \ge \frac{\sqrt{n+1}}{(n-1)!}\frac{1}{e}$.
\end{theorem}

%Note that the bound proven in Theorem \ref{thm:minimizer} is asymptotically optimal, since the ratio between \\ $\frac{\sqrt{n}}{(n-1)!}(\frac{n}{n+1})^{n-1} $ and $\frac{\sqrt{n+1}}{(n-1)!}\frac{1}{e}$ goes to $1$ as $n$ goes to infinity.

%maybe include a section at the end about applications to geometric problems

\section{The technique of probability distributions}\label{sec:prob}

In Section 0 of \cite{webb} it is proven that if $Y_1, Y_2, \dots, Y_{n+1}$ are i.i.d. standard exponential random variables (mean $1$), and $G_{a}(x)$ represents the probability density function of the sum $a_1Y_1 + a_2Y_2 + \dots + a_{n+1}Y_{n+1}$, then the volume of $T(a)$ is given by \begin{equation}\label{eq:prob-interpret}\vol_{n-1}(T(a)) = \frac{\sqrt{n+1}}{(n-1)!}G_{a}(0).\end{equation} In other words, the volume of $T(a)$ is proportional to the density at $0$ of the random variable $a_1Y_1 + a_2Y_2 + \dots + a_{n+1}Y_{n+1}$. So Question \ref{q:main} and Theorem \ref{thm:minimizer} may be converted to the following statements about probability distributions:

\begin{question}\label{q:prob}
    What is the smallest possible value of $G_{a}(0)$, as $a$ varies over all unit vectors in $\RR^{n+1}$ which are perpendicular to $\vec{\mathbf{1}}$?
\end{question}
\begin{theorem}\label{thm:minimizer-prob}
    For each unit vector $a \in \RR^{n+1}$ satisfying $a \perp \vec{\mathbf{1}}$, we have $G_{a}(0) \ge \frac{1}{e}$.
\end{theorem}

Throughout this paper, let $u \in \RR^{n+1}$ denote an arbitrary unit vector (not necessarily satisfying $u \perp \vec{\mathbf{1}}$). Let $G_u(x)$ denote the probability density of the random variable $u_1(Y_1-1)+u_2(Y_2-1)+\dots+u_{n+1}(Y_{n+1}-1)$ (sum of independent \textbf{centered} exponential random variables), and observe that when $u \perp \vec{\mathbf{1}}$ this agrees with the earlier definition of $G_a(x)$. We will prove the following strengthening of Theorem \ref{thm:minimizer-prob}, which may be of independent interest:
\begin{theorem}\label{thm:minimizer-prob-strong}
    For each unit vector $u \in \RR^{n+1}$, not necessarily satisfying $u \perp \vec{\mathbf{1}}$, we have $G_u(0) \ge \frac{1}{e}$. Equality occurs if $u$ is of form $(1,0,0,0,\dots,0)$.
\end{theorem}
Note that this theorem is precisely the same as Theorem \ref{thm:prob-distro}.

\textbf{Note.} \textit{It is true that equality holds in Theorem \ref{thm:minimizer-prob-strong} if and only if $u$ is of form $(1,0,0,0,\dots,0)$, though this will not be proven. Since $(1,0,0,0,\dots,0)$ is not perpendicular to $\vec{\mathbf{1}}$, there is no vector $a$ which achieves equality in Theorem \ref{thm:minimizer-prob}, and hence no central section $T(a)$ achieves equality in Theorem \ref{thm:minimizer}. In other words, Theorem \ref{thm:minimizer} is not sharp, even though Theorem \ref{thm:minimizer-prob-strong} is. We lost some sharpness when we considered the space of all unit vectors $u$ (instead of the space of unit vectors $a$ satisfying $a \perp \vec{\mathbf{1}}$). This is why we cannot quite prove that $a_\mathrm{facet}$ achieves the minimum central section. We can only prove Theorem \ref{thm:main}, which says that $a_\mathrm{facet}$ is within a $1-o(1)$ factor of the minimum central section.}

%We can easily check that Theorem \ref{thm:minimizer-prob-strong} holds in the case when $u$ is of form $(1,0,0,0,\dots,0)$. Indeed, for this $u$ we have that $G_u(x)$ is the probability density of the random variable $Y_1-1$, and thus $G_u(0) = e^{-1}=\frac{1}{e}$. 
Since for $u=(1,0,0,0,\dots,0)$ we have that $G_u(x)$ is the probability density of the random variable $Y_1-1$ with $G_u(0)=\frac{1}{e}$, we henceforth restrict our attention to proving Theorem \ref{thm:minimizer-prob-strong} when $u$ is not of this form; i.e. when $u$ has at least two nonzero entries.
%Note that Theorem \ref{thm:minimizer-prob} is exactly the same as Theorem \ref{thm:prob-distro}.

%In other words, we have converted a question about geometry (Question \ref{q}) to a question about probability distributions (Question \ref{q:prob}).

Let us try to understand the probability distribution $G_u$. For each $1 \le j \le n+1$, define $f_j(x)$ to be the probability density function of the random variable $u_j(Y_j-1)$. Note that the $u_j(Y_j-1)$ are mutually independent. Since $G_{u}$ is defined to be the probability density function of the sum of these independent random variables, we know that $G_{u}$ is the convolution of the $f_j$; i.e. $G_{u} = f_1 * f_2 * \dots * f_{n+1}$. 

%\begin{figure}
%    \centering
%\begin{asy}
%import graph;
%unitsize(2cm);
%pair f(real t){ return (t,exp(-t));}
%path pos = graph(f,0,2);
%draw(pos,linewidth(1.5));
%draw((-2,0)--(0,0)--(0,1),linewidth(1.5));
%draw((-2,0)--(2,0),arrow=Arrows());
%label("$x$",(2.2,0));
%draw((0,0)--(0,1.5),arrow=Arrow());
%label("$f(x)$",(0,1.6));
%draw((-0.1,1)--(0.1,1));
%label("$1$",(-0.2,1));
%\end{asy}
%    \caption{Graph of $f$, which is the probability density function of an exponential distribution with mean $1$.}
%\end{figure}

Define the function $f:\RR \af [0,+\infty)$ via \[f(x) = \begin{cases}e^{-x} &\text{ if } x \ge 0 \\ 0 &\text{ otherwise}. \end{cases}\] Note that $f$ is the probability density function of a standard exponential random variable (mean $1$), so each $Y_j$ has probability density function given by $f$. It follows that $f_j(x) = \frac{1}{|u_j|}f(\frac{x}{u_j}+1)$ for each $j$, and hence we may write 
\begin{equation*}\begin{split}G_{u}(x) &= (f_1 * f_2 * \dots * f_{n+1})(x) \\
&= \int_{x_1+x_2+\dots+x_{n+1}=x} f_1(x_1)f_2(x_2)\dots f_{n+1}(x_{n+1}) \, \dd x_1\, \dd x_2\,\cdots\,\dd x_{n+1} \\
&= \int_{x_1+x_2+\dots+x_{n+1}=x} \prod_{j=1}^{n+1} \frac{1}{|u_j|}f\left(\frac{x_j}{u_j}+1\right) \, \dd x_1\,\dd x_2\,\cdots\,\dd x_{n+1}.\end{split}\end{equation*}

%Note that $f$ is a log-concave function, and thus $f_j$ is a log-concave function of $x$ for each $j$. Since the convolution of log-concave functions is again log-concave, we have that $G$ is log-concave.

\section{The technique of the Fourier inversion formula}\label{sec:fourier}

We will now use the Fourier transform to convert the convolution into a pointwise product. Defining the Fourier transform (Section 2 of \cite{webb}) \[\widehat{G_{u}}(t) = \int_{-\infty}^{+\infty} G_{u}(x) e^{-itx} \, \dd x\] we have \[\widehat{G_{u}}(t) = \prod_{j=1}^{n+1} \widehat{f_j}(t).\] From $\hat{f}(t) = \frac{1}{1+it}$, we calculate $\widehat{f_j}(t) = \frac{e^{iu_jt}}{1+iu_jt}$ and thus \[\widehat{G_u}(t) = \prod_{j=1}^{n+1} \frac{e^{iu_jt}}{1+iu_jt}.\] Note that whenever $u$ has at least two nonzero entries, we have $\widehat{G_u} \in L^1$ and hence the Fourier inversion formula is valid, yielding \[G_{u}(x) = \frac{1}{2\pi} \int_{-\infty}^{+\infty} \widehat{G_u}(t) e^{ixt} \, \dd t = \frac{1}{2\pi} \int_{-\infty}^{+\infty} \left(\prod_{j=1}^{n+1} \frac{e^{iu_jt}}{1+iu_jt} \right) e^{ixt} \, \dd t.\] In particular, plugging in $x=0$ yields \[G_u(0) = \frac{1}{2\pi} \int_{-\infty}^{+\infty} \left(\prod_{j=1}^{n+1} \frac{e^{iu_jt}}{1+iu_jt} \right) \, \dd t.\] If we want to show a lower bound on $G_{u}(0)$ (as in Theorem \ref{thm:minimizer-prob-strong}), it suffices to prove the following inequality: \[\frac{1}{2\pi} \int_{-\infty}^{+\infty} \left(\prod_{j=1}^{n+1} \frac{e^{iu_jt}}{1+iu_jt} \right)  \, \dd t \ge \frac{1}{e}.\] Writing \begin{equation}\label{eq:F-def}F_u(t) \coloneqq \prod_{j=1}^{n+1} \frac{e^{iu_jt}}{1+iu_jt},\end{equation} the inequality that we wish to show becomes the following: \begin{theorem}\label{thm:1-e}
    For each unit vector $u \in \RR^{n+1}$ such that $u$ has at least two nonzero entries, we have \[\frac{1}{2\pi} \int_{-\infty}^{+\infty} F_u(t) \, \dd t \ge \frac{1}{e}.\]
\end{theorem}

\section{Main idea: Changing the contour of integration}\label{sec:change-contour}
%For each unit vector $\hat{a}$ satisfying $\hat{a} \perp \hat{\mathbf{1}}$, we have defined the meromorphic function $F_{\hat{a}}(t)$ via \[F_{\hat{a}}(t) \coloneqq \prod_{j=1}^{n+1}\frac{1}{1+ia_jt}.\] 

%Now, for each unit vector $\hat{u} \in \RR^{n+1}$ (not necessarily satisfying $\hat{u} \perp \hat{\mathbf{1}}$), let us define a meromorphic function $F_{\hat{u}}(t)$. As in Equation (\ref{eq:a-coords}) from Section \ref{sec:def}, write $\hat{u}$ in coordinates as \[\hat{u} = (u_1,u_2,\dots,u_{n+1})\] (so $u_1^2+u_2^2+\dots+u_{n+1}^2 = 1$ and each $u_j$ lies in $[-1,1]$). Then define \begin{equation}\label{eq:F-def}F_{\hat{u}}(t) \coloneqq \prod_{j=1}^{n+1} \frac{e^{iu_jt}}{1+iu_jt}.\end{equation} Note that this agrees with the earlier definition of $F_{\hat{a}}(t)$ (since $a_1+a_2+\dots+a_{n+1}=0$).

%Just as $F_{\hat{a}}(t)$ corresponds to the Fourier transform of the density of $a_1Y_1+a_2Y_2+\dots+a_{n+1}Y_{n+1}$ (see Section \ref{sec:prob}), we may interpret $F_{\hat{u}}(t)$ as the Fourier transform of the density of $u_1(Y_1-1)+u_2(Y_2-1)+\dots+u_{n+1}(Y_{n+1}-1)$.

\begin{figure}
    \centering
    \includegraphics{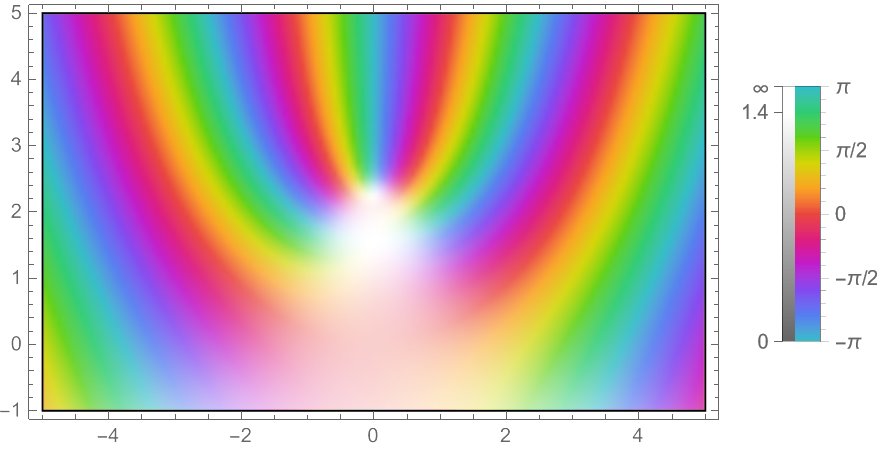}
    \caption{A plot of $F_u(t)$ on the Argand plane when $u = (\sqrt{0.42},\sqrt{0.38},\sqrt{0.20})$. The horizontal axis denotes the real part of $t$, and the vertical axis denotes the imaginary part of $t$. The brightness represents the modulus of $F_u(t)$, while the hue represents the argument of $F_u(t)$ (see the legends to the right of the plot). Note the poles along the imaginary axis. The contour $\gamma_u$ will be chosen such that $F_u$ is red along $\gamma_u$. \\ This image was produced using Wolfram Mathematica Student Edition.}
    \label{fig:F-u}
\end{figure}

%To show Equation (\ref{eq:F-ineq}), we will show the following stronger statement:

In this section we will showcase our main new idea, which is to change the contour of integration.

As noted in Section \ref{sec:real-intro}, it is difficult to bound the integral $\int_{-\infty}^{+\infty} F_u(t) \, \dd t$ as written, because $F_u$ oscillates on the real axis (see Figure \ref{fig:F-u}). We will instead find a contour $\gamma_{u}$ in the complex plane (the shape of the contour will depend on the choice of $u$), along which $F_{u}$ takes positive real values (Proposition \ref{prop:real-values}); finally we will argue that the integral $\int_{-\infty}^{+\infty} F_u(t) \, \dd t$ equals the contour integral $\int_{\gamma_u} F_u(t) \, \dd t$ (see Equation (\ref{eq:move-contour})), allowing us to change the integral in Theorem \ref{thm:1-e} to an easier integral. 

The rest of this section will be dedicated to finding and defining $\gamma_u$.

As in Section \ref{sec:prob}, let $u \in \RR^{n+1}$ be a fixed unit vector. We want to find a function $y_u:\RR\af\RR$, depending on $u$, such that the graph of $y_u$ (i.e. the set $\left\{x+iy_u(x) \bigm\vert x \in \RR \right\}$) is exactly the contour $\gamma_u$ that we seek. In order for this to be the case, we need \[\arg F_u(x+iy_u(x)) \equiv 0 \pmod{2\pi}\] for each $x$, where $\arg z$ denotes the argument of a nonzero complex number $z$. Using the definition of $F_u(t)$ (Equation (\ref{eq:F-def})), we can write \begin{equation}\label{eq:arg-computation}\begin{split}\arg F_u(x+iy) &= \arg \prod_{j=1}^{n+1} \frac{e^{iu_j(x+iy)}}{1+iu_j(x+iy)} \\ 
&= x\left(\sum_{j=1}^{n+1} u_j\right) - \sum_{j=1}^{n+1} \arg (1+iu_j(x+iy))\\
&= x\left(\sum_{j=1}^{n+1} u_j\right) - \sum_{j=1}^{n+1} \arg (1-u_jy+iu_jx)\\
&\equiv x\left(\sum_{j=1}^{n+1} u_j\right) - \sum_{j=1}^{n+1} \arctan \frac{u_jx}{1-u_jy} \pmod{2\pi}
\end{split}\end{equation} so all we need to do is to choose a continuous function $y_u$ such that \begin{equation}\label{eq:angle-equation} x\left(\sum_{j=1}^{n+1} u_j\right) - \sum_{j=1}^{n+1} \arctan \frac{u_jx}{1-u_jy_u(x)} \equiv 0 \pmod{2\pi}\end{equation} holds for all $x$. For any fixed $x$, there may be multiple values of $y_u(x)$ which satisfy Equation (\ref{eq:angle-equation}) (due to the fact that that equation is modulo $2\pi$), but if we additionally stipulate that $y_u$ is continuous and satisfies $y_u(0) = 0$, then there is only one way to choose $y_u$. We will describe how to choose $y_u$ below.

Let $x,y$ be arbitrary real numbers satisfying $x > 0$. We will adopt the convention that the $\arccot$ function is a continuous function from $\RR$ to $\RR$ that has range $(0,\pi)$. (Some authors define $\arccot \theta$ to have a jump discontinuity at $\theta = 0$; we will not use their convention here.) Assume that all $u_j$ are nonzero (since we can just discard any zero entries). 

%\begin{figure}
%    \centering
%\begin{asy}
%import graph;
%unitsize(0.7cm);
%pair f(real t){return (t,pi/2-atan(t));}
%path pos = graph(f,-4,4);
%draw(pos,linewidth(1.5));
%//draw((-2,0)--(0,0)--(0,1),linewidth(1.5));
%draw((-4,0)--(4,0),arrow=Arrows());
%label("$x$",(4.3,0));
%draw((0,0)--(0,pi),arrow=Arrow());
%label("arccot $x$",(0,pi+0.2));
%//draw((-0.1,1)--(0.1,1));
%//label("$1$",(-0.2,1));
%\end{asy}
%    \caption{A plot of $\arccot x$. Our convention is that this is a continuous function on all of $\RR$ which has range $(0,\pi)$.}
    %\label{fig:enter-label}
%\end{figure}

For each $1 \le j \le n+1$, we define $\alpha_j = \alpha_j(x,y;u) \coloneqq \arccot \frac{\frac{1}{u_j}-y}{x}$ if $u_j > 0$, and we define $\beta_j = \beta_j(x,y;u) \coloneqq \arccot \frac{-(\frac{1}{u_j}-y)}{x}$ if $u_j < 0$. We will not define $\alpha_j$ for $u_j < 0$, nor will we define $\beta_j$ for $u_j > 0$. (We may suppress the dependence on $x,y,u$, thus writing merely $\alpha_j$ for what should technically be written as $\alpha_j(x,y;u)$, if it is clear from context.) We define the function \begin{equation}\label{eq:Phi-u-1} \Phi_u(x,y) \coloneqq x\sum_{j=1}^{n+1}u_j - \sum_{j:u_j>0}\alpha_j + \sum_{j:u_j<0} \beta_j, \end{equation} and the computation in Equation (\ref{eq:arg-computation}) implies \begin{equation}\label{eq:Phi-u-2}\Phi_u(x,y) \equiv \arg F_u(x+iy) \pmod{2\pi}\end{equation} for all $(x,y)$ satisfying $x>0$. 

The advantage of defining $\Phi_u$ is that $\Phi_u$ takes values in $\RR$, whereas $\arg F_u(x+iy)$ is only defined modulo $2\pi$. So we can say that $\Phi_u(x,y)$ is a real-valued function that is a \emph{lifting} of the function $\arg F_u(x+iy)$. In other words, $\Phi_u(x,y)$ is a real-valued function that always is equivalent to the argument of $F_u(x+iy)$ modulo $2\pi$. 

We can now say that the pair $(x,y)$ \textbf{has good angle} if $\Phi_u(x,y) = 0$, this equality holding in $\RR$ (instead of just $\RR/2\pi\ZZ$).

\begin{figure}
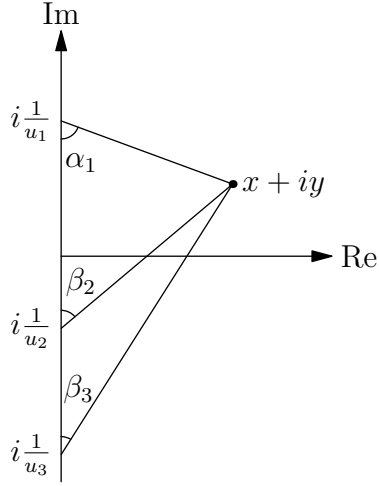

    \centering
\begin{asy}
unitsize(1.2cm);
draw((0,0)--(3,0),arrow=Arrow());label("Re",(3.3,0));
draw((0,-2.5)--(0,2.5),arrow=Arrow());label("Im",(0,2.7));
pair P = (1.9,0.8);
dot(P);
label("$x+iy$",P,E);
real u1 = 1.5;
real u2 = -0.8;
real u3 = -2.2;
draw(P--(0,u1)); label("$i\frac{1}{u_1}$",(0,u1),W);
draw(P--(0,u2)); label("$i\frac{1}{u_2}$",(0,u2),W);
draw(P--(0,u3)); label("$i\frac{1}{u_3}$",(0,u3),W);
real rad = 0.2;
draw(arc((0,u1),(0,u1-rad),point(circle((0,u1),rad),intersect(circle((0,u1),rad),P--(0,u1))[0]))); label("$\alpha_1$",(0,u1),2*SE+2*S);
draw(arc((0,u2),intersectionpoint(circle((0,u2),rad),P--(0,u2)),(0,u2+rad))); label("$\beta_2$",(0,u2),2*NE+2*N);
draw(arc((0,u3),intersectionpoint(circle((0,u3),rad),P--(0,u3)),(0,u3+rad))); label("$\beta_3$",(0,u3),2*NE+4*N);
\end{asy}
    \caption{For each $j$, we define $\alpha_j \coloneqq \arccot \frac{\frac{1}{u_j}-y}{x}$ if $u_j > 0$, and we define $\beta_j \coloneqq \arccot \frac{-(\frac{1}{u_j}-y)}{x}$ if $u_j < 0$.}
\end{figure}

Let $m_+, m_-$ denote the number of indices $j$ such that $u_j$ is positive (respectively, negative). So $m_+ + m_- = n+1$. Since $\Phi_u$ is a strictly decreasing function of $y$, we know that: \begin{enumerate}
    \item If $x$ satisfies $-m_-\pi < x\sum_{j=1}^{n+1}u_j < m_+ \pi$, then there exists a unique $y$ such that $(x,y)$ has good angle. In this case we will set $y_u(x) = y$.
    \item If $x$ does not satisfy $-m_-\pi < x\sum_{j=1}^{n+1}u_j < m_+ \pi$, then there does not exist $y$ such that $(x,y)$ has good angle.
\end{enumerate}

Thus we have found a function $y_u$ defined on the domain \[D_u \coloneqq \left\{x \biggm\vert x>0 \text{ and } {-m_-\pi} < x\sum_{j=1}^{n+1}u_j < m_+ \pi\right\}.\] (Observe that $D_u$ contains the interval $(0,\pi)$.) By definition, \begin{equation}\label{eq:Phi-x-y}\Phi_u(x,y_u(x)) = 0\end{equation} for all $x \in D_u$. In other words, $y_u$ is contained in the \textbf{zero locus} of $\Phi_u$.

We can obtain that $y_u$ is $\mathcal{C}^\infty$ (in fact, real analytic) on $D_u$ using the Implicit Function Theorem ($\Phi_u$ is a real analytic function of $(x,y)$ which is also strictly decreasing with respect to $y$). 

We may also show that $\lim_{x \searrow 0} y_u(x) = 0$. To do this, let $\epsilon \in (0,1)$ be arbitrary and fixed. If $j$ is such that $u_j>0$, we note that $\frac{1}{u_j}-\epsilon > \frac{1}{1}-1 = 0$, so $(\arccot \frac{\frac{1}{u_j}-\epsilon}{x})/(\frac{x}{\frac{1}{u_j}-\epsilon})$ goes to $1$ as $x \searrow 0$. Similarly, if $j$ is such that $u_j < 0$, we have $(\arccot \frac{-(\frac{1}{u_j}-\epsilon)}{x})/(\frac{x}{-(\frac{1}{u_j}-\epsilon)})$ goes to $1$ as $x \searrow 0$. So \[\frac{1}{x}\left( \sum_{j : u_j > 0} \arccot \frac{\frac{1}{u_j}-\epsilon}{x} - \sum_{j : u_j < 0} \arccot \frac{-(\frac{1}{u_j}-\epsilon)}{x} \right) \af \sum_{j : u_j > 0} \frac{1}{\frac{1}{u_j}-\epsilon} - \sum_{j : u_j < 0} \frac{1}{-\left(\frac{1}{u_j}-\epsilon\right)} = \sum_{j=1}^{n+1} \frac{1}{\frac{1}{u_j}-\epsilon}\] as $x \searrow 0$. In other words, \[\lim_{x \searrow 0} \frac{1}{x}\left( \sum_{j:u_j>0}\alpha_j(x,\epsilon;u) - \sum_{j:u_j<0}\beta_j(x,\epsilon;u) \right) = \sum_{j=1}^{n+1} \frac{1}{\frac{1}{u_j}-\epsilon}.\] Since $\sum_{j=1}^{n+1} \frac{1}{\frac{1}{u_j}-\epsilon} > \sum_{j=1}^{n+1} u_j$, we see that when $x$ is sufficiently close to $0$, we have $\Phi_u(x,\epsilon) < 0$. Using the fact that $\Phi_u$ is strictly decreasing with respect to $y$, we obtain $y_u(x) < \epsilon$ when $x$ is sufficiently close to $0$. Similarly, we can show that $y_u(x) > -\epsilon$ when $x$ is sufficiently close to $0$. Since $\epsilon$ was arbitrary, this proves $\lim_{x \searrow 0} y_u(x) = 0$, as desired.

Thus $y_u$ extends continuously to the point $x=0$, by setting $y_u(0) \coloneqq 0$. 

Let us extend $y_u$ by reflection about the $y$-axis. Explicitly, we define an expanded domain \[E_u \coloneqq \left\{x \bigm\vert x=0 \text{ or } x \in D_u \text{ or } {-x} \in D_u\right\}\] and extend $y_u$ to this expanded domain by setting $y_u(x) = y_u(-x)$ whenever $-x \in D_u$. So $y_u$ will be a continuous function on this expanded domain, and is in fact $\mathcal{C}^\infty$ away from the point $x=0$.

We will now show that $y_u$ is $\mathcal{C}^\infty$ on the entirety of $E_u$. Observe that for $x>0$ and $y \in (-1,1)$, we have \begin{equation*}\begin{split}\frac{1}{x}\Phi_u(x,y) &= \sum_{j=1}^{n+1}u_j - \sum_{j:u_j>0}\frac{1}{x}\arccot \frac{\frac{1}{u_j}-y}{x} + \sum_{j:u_j<0}\frac{1}{x}\arccot \frac{-(\frac{1}{u_j}-y)}{x}\\
&= \sum_{j=1}^{n+1}u_j - \sum_{j:u_j>0}\frac{1}{x}\arctan \frac{x}{\frac{1}{u_j}-y} + \sum_{j:u_j<0}\frac{1}{x}\arctan \frac{x}{-(\frac{1}{u_j}-y)}
\end{split}\end{equation*} since $\arccot \theta = \arctan (1/\theta)$ whenever $\theta>0$. Now we let $\phi$ denote the real analytic function given by $\phi(\theta) \coloneqq (\arctan\theta) / \theta$, so we have \begin{equation*}
    \begin{split}
        \frac{1}{x}\Phi_u(x,y) &= \sum_{j=1}^{n+1}u_j - \sum_{j:u_j>0}\frac{1}{\frac{1}{u_j}-y}\phi\left( \frac{x}{\frac{1}{u_j}-y}\right) - \sum_{j:u_j<0}\frac{1}{\frac{1}{u_j}-y}\phi\left( \frac{x}{-(\frac{1}{u_j}-y)}\right)\\
    \end{split}
\end{equation*} and since $\phi$ is an even function, we can write \begin{equation*}
    \begin{split}
        \frac{1}{x}\Phi_u(x,y) &= \sum_{j=1}^{n+1}u_j - \sum_{j=1}^{n+1}\frac{1}{\frac{1}{u_j}-y}\phi\left( \frac{x}{\frac{1}{u_j}-y}\right). \\
    \end{split}
\end{equation*} The right-hand-side can be extended to a real analytic function $\psi_u(x,y)$ on all of $\RR \times (-1,1)$. Moreover, we compute \[\left(\frac{\partial}{\partial y}\psi_u \right)(0,0) = -\sum_{j=1}^{n+1} \frac{1}{\left(\frac{1}{u_j}\right)^2} -\sum_{j=1}^{n+1}\frac{1}{\left(\frac{1}{u_j}\right)}\cdot 0 = -\sum_{j=1}^{n+1}u_j^2 = -1 \not= 0,\] so by the Implicit Function Theorem we know that near $(0,0)$, the zero locus of $\psi_u$ is the graph of a real analytic function. But this zero locus is exactly the graph of $y_u$, so we obtain that $y_u$ is a real analytic function of $x$ at the point $x=0$, as desired.

\begin{figure}
    \centering
\includegraphics[scale=0.7]{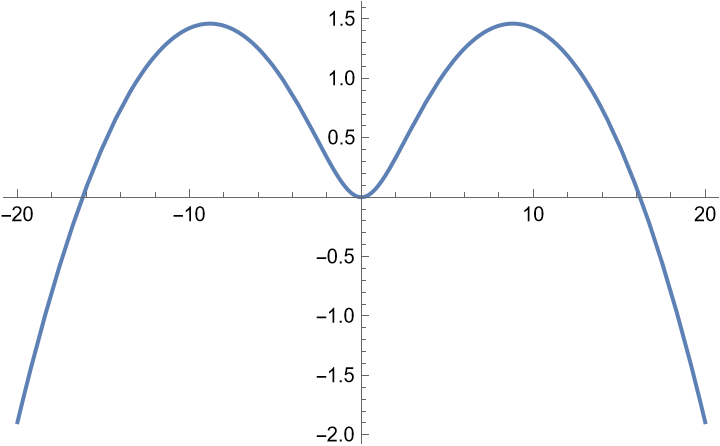}
    \caption{A plot of $y_u(x)$ when $u = (\sqrt{0.62},-\sqrt{0.19},-\sqrt{0.19})$, for $-20 < x < 20$. \\ This image was produced using Wolfram Mathematica Student Edition.}
\end{figure}

\begin{remark}
    What happens for other probability distributions (other than the exponential distribution that we have been considering)? Are we still able to define a smooth $y_u$ on which $F_u$ takes positive real values?
\end{remark}

We have our function $y_u$ which is $\mathcal{C}^\infty$ on $E_u$. By construction, we have that $y_u$ satisfies Equation (\ref{eq:angle-equation}) for all $x \in E_u$. Now if we define \[\gamma_u \coloneqq \left\{x+iy_u(x)\bigm\vert x\in E_u \right\}\] to be the graph of $y_u$, then $\gamma_u$ is exactly the contour that we seek. Indeed, the key property of $\gamma_u$ is the following:

\begin{proposition}\label{prop:real-values}
For each $t \in \gamma_u$, we have that $F_u(t)$ is a positive real number.
\end{proposition}
This proposition follows from Equation (\ref{eq:Phi-u-2}).

Now that we have our contour $\gamma_u$, we would like to say the following (in order to change the integral in Theorem \ref{thm:1-e} to an easier integral):
\begin{theorem}\label{thm:change-integral}
    For each unit vector $u \in \RR^{n+1}$ such that $u$ has at least two nonzero entries, we have that the integral \[\frac{1}{2\pi}\int_{-\infty}^{+\infty} F_u(t) \, \dd t\] equals the integral \[\frac{1}{2\pi} \int_{\gamma_u} F_{u}(t) \, \Real(\dd t) = \frac{1}{2\pi} \int_{s \in E_{u}} F_{u}(s+iy_{u}(s)) \, \dd s.\]
\end{theorem}
Note that the latter integral is well-defined because the integrand is always a positive real number on the path of integration. 

To show Theorem \ref{thm:change-integral}, we will use the fact that \emph{the integral of a meromorphic function does not depend on the contour, as long as one contour can be moved to the other without passing over any poles}. \textbf{\textit{This is the main idea of our paper: moving the contour of integration.}} We will approximate both integrals in Theorem \ref{thm:change-integral} by integrals over bounded contours, and then we will use Cauchy's Integral Theorem to shift the contour of integration. We must use an upper bound on the magnitude of $F_u(t)$; this is because when shifting the contour, we will incur a small error at each of the fringes, and we need to know that this error goes to zero as we take the fringes to be further and further away from the origin. 

\begin{proof}[Proof of Theorem \ref{thm:change-integral}]

To upper bound the magnitude of $F_u(t)$, we now use the fact that $u$ has at least two nonzero entries. Hence there exists a constant $C > 0$, depending only on $u$, such that for all pairs $(x,y) \in \RR^2$ with $x \not= 0$, we have \begin{equation}\label{eq:F-bound}|F_u(x+iy)| \le C\exp\left(-y\sum_{j=1}^{n+1} u_j\right)\left/\left(x^2+\min_{j=1}^{n+1} \left(\frac{1}{u_j}-y\right)^2\right)\right. .\end{equation} (In particular, the integral $\frac{1}{2\pi}\int_{-\infty}^{+\infty} F_u(t) \, \dd t$ in Theorem \ref{thm:1-e} is well-defined.)

We split into cases based on the sign of $\sum_{j=1}^{n+1}u_j$.

\textbf{Case I:} $\sum_{j=1}^{n+1} u_j > 0$. \\
In this case, $E_u = (-m_+\pi / \sum_{j=1}^{n+1} u_j, m_+\pi / \sum_{j=1}^{n+1} u_j)$ is a finite interval. Now let us choose $x_1 \in D_u$ and $x_2 > x_1$. Let $K_{x_1,x_2}$ denote the path consisting of the line segment from $x_1+iy_u(x_1)$ to $x_2 + iy_u(x_1)$, concatenated with the line segment from $x_2 + iy_u(x_1)$ to $x_2$; we will use $K_{x_1,x_2}$ as a fringe to relate the two contours ($\gamma_u$ and the real axis). Equation (\ref{eq:F-bound}) implies that $|F_u(x+iy)| \le C\exp\left(-y\sum_{j=1}^{n+1} u_j\right)/x^2$ for all $x\not= 0$. By breaking $K_{x_1,x_2}$ into its constituent line segments, we obtain \[\left|\int_{K_{x_1,x_2}} F_{u}(t) \, \dd t\right| \le C\exp\left(-y_{u}(x_1)\sum_{j=1}^{n+1} u_j\right)\big/x_1 + (C/x_2^2)\left|\int_0^{y_{u}(x_1)} \exp\left(-t\sum_{j=1}^{n+1} u_j \right) \, \dd t\right|.\] 

Now note that $\lim_{x \nearrow m_+\pi / \sum_{j=1}^{n+1} u_j} y_u(x) = +\infty$. If $x_1$ is chosen sufficiently close to the right endpoint of $D_u$, we will have $y_u(x_1) > 0$, and then we can write \[\left|\int_{K_{x_1,x_2}} F_u(t) \, \dd t\right| \le C\exp\left(-y_{u}(x_1)\sum_{j=1}^{n+1} u_j\right)\big/x_1 + (C/x_2^2)\int_0^{+\infty} \exp\left(-t\sum_{j=1}^{n+1} u_j \right) \, \dd t.\] Thus \[\lim_{\substack{x_1 \nearrow m_+\pi / \sum_{j=1}^{n+1} u_j \\ x_2 \nearrow +\infty}} \left|\int_{K_{x_1,x_2}} F_{u}(t) \, \dd t\right| = 0;\] i.e. the integral along the fringe $K_{x_1,x_2}$ goes to zero, from which we conclude \[\lim_{x \nearrow m_+\pi / \sum_{j=1}^{n+1} u_j} \int_{-x}^x F_{u}(s+iy_{u}(s)) \, \left(\dd s + i \frac{\dd y_{u}(s)}{\dd s} \, \dd s \right) = \int_{-\infty}^{+\infty} F_{u}(t) \, \dd t\] by Cauchy's Integral Theorem (this step is where we have \textbf{\textit{shifted the contour}} from the real axis to $\gamma_u$).

We are not done yet, since the integrand on the left-hand-side of the above equation still contains an extra term of $F_u(s+iy_u(s)) \cdot i \frac{\dd y_u(s)}{\dd s} \, \dd s$. Since $y_u(x)$ is an even function of $x$, we know that $\frac{\dd y_{u}(x)}{\dd x}$ is an odd function of $x$. Additionally, $F_{u}(x+iy_{u}(x))$ is an even function of $x$, so $\int_{-x}^x F_{u}(s+iy_{u}(s))\frac{\dd y_{u}(s)}{\dd s} \, \dd s = 0$ and we may delete the extra term from the integrand: \[\lim_{x \nearrow m_+\pi / \sum_{j=1}^{n+1} u_j} \int_{-x}^x F_{u}(s+iy_{u}(s)) \, \dd s = \int_{-\infty}^{+\infty} F_{u}(t) \, \dd t,\] from which it follows that \begin{equation}\label{eq:move-contour}\int_{s \in E_{u}} F_{u}(s+iy_{u}(s)) \, \dd s = \int_{-\infty}^{+\infty} F_{u}(t) \, \dd t.\end{equation} This proves Theorem \ref{thm:change-integral}, as desired. This concludes Case I.

\textbf{Case II:} $\sum_{j=1}^{n+1} u_j < 0$. \\
This case is similar to Case I.

\textbf{Case III:} $\sum_{j=1}^{n+1} u_j = 0$. \\
In this case, we have that $E_u$ is the entire real line $\RR$. Then Equation (\ref{eq:F-bound}) becomes \[|F_u(x+iy)| \le C\left/\left(x^2+\min_{j=1}^{n+1} \left(\frac{1}{u_j}-y\right)^2\right)\right. \] for all $x \not= 0$. In particular, we have $|F_u(x+iy)| \le C/x^2$. For convenience, let $M > 0$ be large enough so that $1/u_j \in (-M,M)$ for all $j$, so that we have $|F_u(x+iy)| \le C/(x^2+(|y|-M)^2)$ if $|y| \ge M$. For any $x>0$, we may define the path $K_x$ to be the line segment from $x+iy_u(x)$ to $x$. As before, we will use $K_x$ as a fringe to relate the two contours ($\gamma_u$ and the real axis). We obtain \begin{equation*}\begin{split}\left|\int_{K_x} F_u(t) \, \dd t \right| &\le \int_{x-i\infty}^{x+i\infty} |F_u(t)| \, |\dd t| \\
&= \int_{x-i(M+1)}^{x+i(M+1)} |F_u(t)| \, |\dd t| + \int_{x-i\infty}^{x-i(M+1)} |F_u(t)| \, |\dd t| + \int_{x+i(M+1)}^{x+i\infty} |F_u(t)| \, |\dd t|\\
&\le (C/x^2)(2\cdot (M+1)) + \int_1^{+\infty} C/(x^2+s^2) \, \dd s + \int_1^{+\infty} C/(x^2+s^2) \, \dd s\end{split}\end{equation*} where the first term comes from the bound $|F_{u}(x+iy)| \le C/x^2$ for $|y| \le M+1$, and the other two terms come from the bound $|F_{u}(x+iy)| \le C/(x^2+(|y|-M)^2)$ for $|y| \ge M+1$. Thus \[\lim_{x \nearrow +\infty} \left|\int_{K_x} F_{u}(t) \, \dd t \right| = 0;\] i.e. the integral along the fringe $K_x$ goes to zero, from which we conclude \[\lim_{x \nearrow +\infty} \int_{-x}^x F_{u}(s+iy_{u}(s)) \, \left(\dd s + i \frac{\dd y_{u}(s)}{\dd s} \, \dd s \right) = \int_{-\infty}^{+\infty} F_{u}(t) \, \dd t\] by Cauchy's Integral Theorem (this step is where we have \textbf{\textit{shifted the contour}} from the real axis to $\gamma_u$). Arguing as in Case I, we note that $\int_{-x}^x F_{u}(s+iy_{u}(s))\frac{\dd y_{u}(s)}{\dd s} \, \dd s = 0$, so we may delete the $F_u(s+iy_u(s)) \cdot i \frac{\dd y_u(s)}{\dd s} \, \dd s$ term from the integrand, resulting in \[\lim_{x \nearrow +\infty} \int_{-x}^x F_{u}(s+iy_{u}(s)) \, \dd s = \int_{-\infty}^{+\infty} F_{u}(t) \, \dd t,\] which proves Theorem \ref{thm:change-integral}, as desired. This concludes Case III.

\end{proof}

We have proven that Theorem \ref{thm:change-integral} holds in all cases. This allows us to change the integral in Theorem \ref{thm:1-e}. We have successfully applied our main idea: \textbf{\textit{changing the contour of integration}}.

It remains to show that $\frac{1}{2\pi} \int_{s \in E_{u}} F_{u}(s+iy_{u}(s)) \, \dd s \ge \frac{1}{e}$ for each $u$ with at least two nonzero entries. Defining \[\tilde{F}_u(x) \coloneqq F_u(x+iy_u(x)),\] we wish to show the following:

\begin{theorem}
    For each unit vector $u \in \RR^{n+1}$ with at least two nonzero entries, we have \[\frac{1}{2\pi} \int_{s \in E_{u}} \tilde{F}_{u}(s) \, \dd s \ge \frac{1}{e}.\] 
\end{theorem} 

This is a special case of the following theorem:

\begin{theorem}\label{thm:1-e-2}
    For each unit vector $u \in \RR^{n+1}$ (possibly of form $(1,0,0,0,\dots,0)$), we have \[\frac{1}{2\pi} \int_{s \in E_{u}} \tilde{F}_{u}(s) \, \dd s \ge \frac{1}{e}.\] Equality occurs if $u$ is of form $(1,0,0,0,\dots,0)$.
\end{theorem} 

We will henceforth focus on proving Theorem \ref{thm:1-e-2}.

\section{A pointwise bound via differential equations}\label{sec:pointwise}

Let us first prove the equality case of Theorem \ref{thm:1-e-2}. For convenience, define the unit vector $v \coloneqq (1) \in \RR^1$, and let $y_v,\gamma_v,D_v,E_v,F_v, \tilde{F}_v$ be defined by substituting $v$ for $u$ in the definitions of $y_u, \gamma_u, D_u, E_u, F_u, \tilde{F}_u$ (in Equation (\ref{eq:F-def}) and in Section \ref{sec:change-contour}). Note that $E_v = (-\pi,\pi)$. We calculate explicitly $F_{v}(t) = e^{it}/(1+it)$ and $y_{v}(x) = 1-x\cot x$ when $x \in (-\pi,\pi)$. It follows that $|F_{v}(x+iy)| = e^{-y}/\sqrt{(1-y)^2+x^2}$ whenever $(x,y) \not= (0,1)$. 

%\begin{figure}
%    \centering
%    \includegraphics{simplex-sections/simplex-sections-fourier-v7-plot-y-v.png}
%    \caption{A plot of $y_{v}(x)$ for $-\pi<x<\pi$. \\ This image was produced using Wolfram Mathematica Student Edition.}
%\end{figure}

In order to prove the equality case of Theorem \ref{thm:1-e-2}, we wish to show that \begin{equation}\label{eq:1-eq}\frac{1}{2\pi} \int_{-\pi}^{\pi} \tilde{F}_v(s) \, \dd s = \frac{1}{e}.\end{equation} We will use Cauchy's Residue Theorem. For $x \in (0,\pi)$, let $P_x$ denote the path which is the line segment from $-x+iy_v(x)$ to $x + iy_v(x)$. We have \[\left|\int_{P_x} F_v(t) \, \dd t \right| \le 2x \cdot e^{-y_v(x)}/(y_v(x) - 1) \le 2\pi \cdot e^{-y_v(x)}/(y_v(x) - 1)\] whenever $x$ is such that $y_v(x) > 1$. Since $\lim_{x \nearrow \pi} y_v(x) = +\infty$, we have \[\lim_{x \nearrow \pi} \left|\int_{P_x} F_v(t) \, \dd t \right| = 0.\] Thus we can write \begin{equation*}
    \begin{split}
        \frac{1}{2\pi} \int_{-\pi}^{\pi} \tilde{F}_v(s) \, \dd s &= \frac{1}{2\pi} \int_{-\pi}^{\pi} F_v(s+iy_v(s)) \, \dd s \\ 
        &= \lim_{x \nearrow \pi}\frac{1}{2\pi}\int_{-x}^{x} F_v(s+iy_v(s)) \, \dd s \\ 
        &= \lim_{x \nearrow \pi}\frac{1}{2\pi}\int_{-x}^{x} F_v(s+iy_v(s)) \, \left(\dd s + i\frac{\dd y_v(s)}{\dd s} \, \dd s \right) \\ 
        &= \lim_{x \nearrow \pi}\frac{1}{2\pi}\int_{\text{follow }\gamma_v\text{ from } -x+iy_v(-x) \text{ to } x + iy_v(x)} F_v(t) \, \dd t \\ 
        &= \lim_{x \nearrow \pi}\frac{1}{2\pi}\int_{\substack{\text{follow }\gamma_v\text{ from } -x+iy_v(-x) \text{ to } x + iy_v(x) \\ \text{then follow } P_x \text{ from } x+iy_v(x) \text{ back to } -x + iy_v(-x)}} F_v(t) \, \dd t. \\ 
    \end{split}
\end{equation*} Since $F_v(t)$ has a simple pole at $t=i$, we can use the Cauchy Residue Theorem to obtain \[\frac{1}{2\pi} \int_{-\pi}^{\pi} \tilde{F}_v(s) \, \dd s = \lim_{x \nearrow \pi}\frac{1}{2\pi} 2\pi i \cdot \Res(F_v,i) = \lim_{x \nearrow \pi}e^{-1} = e^{-1},\] proving Equation (\ref{eq:1-eq}), and thus proving the equality case of Theorem \ref{thm:1-e-2}.

Now that we know that $v$ achieves the equality case in Theorem \ref{thm:1-e-2}, we can prove Theorem \ref{thm:1-e-2} just by showing $\int_{s \in E_u} \tilde{F}_u(s) \, \dd s \ge \int_{s \in E_v} \tilde{F}_v(s) \, \dd s$ for each $u$. Note that for any $u$, we have $E_v = (-\pi,\pi) \subseteq E_{u}$. Thus we can write \[\int_{s \in E_u} \tilde{F}_u(s) \, \dd s \ge \int_{s \in E_v} \tilde{F}_u(s) \, \dd s\] since the integrand is positive. It follows that it suffices to prove $\int_{s \in E_v} \tilde{F}_u(s) \, \dd s \ge \int_{s \in E_v} \tilde{F}_v(s) \, \dd s$, which would be implied by the following pointwise bound:

\begin{theorem}\label{thm:pointwise}
    For each unit vector $u \in \RR^{n+1}$, and for any $x \in (-\pi,\pi)$, we have $\tilde{F}_{u}(x) \ge \tilde{F}_{v}(x)$.
\end{theorem}

The remainder of this section will be devoted to proving Theorem \ref{thm:pointwise}.

We first note that we only have to prove Theorem \ref{thm:pointwise} when $x \in (0,\pi)$ (since $\tilde{F}_u(0) = \tilde{F}_v(0) = 1$, and since $\tilde{F}_u(x)$, $\tilde{F}_v(x)$ are both even functions of $x$).

It suffices to show that $\log \tilde{F}_u(x) \ge  \log \tilde{F}_v(x)$ for all $x \in (0,\pi)$. To do this, we will show that \begin{equation}\label{eq:log-ineq}\frac{\dd}{\dd x}\left(\log \tilde{F}_u(x)\right) \ge \frac{\dd}{\dd x}\left(\log \tilde{F}_v(x)\right)\end{equation} holds for every $x \in (0,\pi)$; this suffices because $\tilde{F}_u(0) = \tilde{F}_v(0) = 1$ as noted above.

From now on, we suppress the variable $x$ when writing $y_u$, so we write simply $y_u$ and $y_v$ instead of $y_u(x)$ and $y_v(x)$. Now we will give a differential equation that $y_u$ must satisfy. For all $x \in D_u$, we have $\Phi_u(x,y_u) = 0$, which implies \[\sum_{j:u_j>0}\arccot \frac{\frac{1}{u_j}-y_u}{x} - \sum_{j:u_j<0} \arccot \frac{-(\frac{1}{u_j}-y_u)}{x} = x \sum_{j=1}^{n+1} u_j  .\] Differentiating with respect to $x$, we obtain \[\sum_{j=1}^{n+1}\left(-\frac{-\frac{\frac{1}{u_j}-y_u}{x^2}+\frac{-y_u'}{x}}{1+\left(\frac{\frac{1}{u_j}-y_u}{x}\right)^2}\right) = \sum_{j=1}^{n+1} u_j\] where $y_u'$ denotes $\frac{\dd}{\dd x}y_u$. We can further simplify to \[\sum_{j=1}^{n+1} \frac{\frac{1}{u_j}-y_u+x y_u'}{x^2+\left(\frac{1}{u_j}-y_u\right)^2} = \sum_{j=1}^{n+1} u_j.\] Rearranging terms yields \begin{equation}\label{eq:diff-y}\sum_{j=1}^{n+1}\frac{x}{x^2+\left(\frac{1}{u_j}-y_u\right)^2} \cdot y_u' = \sum_{j=1}^{n+1}\frac{-y_u+ u_j \cdot (x^2+ y_u^2)}{x^2+\left(\frac{1}{u_j}-y_u\right)^2}\end{equation} which is our desired differential equation, valid for all $x \in D_u$. In particular, by plugging in $v$ for $u$, we obtain \begin{equation*}
    y_v' = \frac{-y_v + (x^2 + y_v^2)}{x},
\end{equation*} valid for all $x \in (0,\pi)$.

From Equation (\ref{eq:diff-y}) and the fact that each $u_j \in [-1,1]$, we obtain \begin{equation*}
    \begin{split}
        \sum_{j=1}^{n+1}\frac{x}{x^2+\left(\frac{1}{u_j}-y_u\right)^2} \cdot y_u' &\le \sum_{j=1}^{n+1}\frac{-y_u+ 1 \cdot (x^2+ y_u^2)}{x^2+\left(\frac{1}{u_j}-y_u\right)^2} \\
        &= \sum_{j=1}^{n+1}\frac{-y_u+ (x^2+ y_u^2)}{x^2+\left(\frac{1}{u_j}-y_u\right)^2} \\
    \end{split}
\end{equation*} from which it follows that \begin{equation}\label{eq:u-diff}y_u' \le \frac{-y_u + (x^2 + y_u^2)}{x}\end{equation} for all $x \in D_u$.

We must first prove an auxiliary lemma relating $y_u$ and $y_v$:

\begin{lemma}\label{lem:vuv-bound}
    For each $x \in (-\pi,\pi)$, we have $-y_v \le y_u \le y_v$.
\end{lemma}
\begin{proof}
    It suffices to prove that $y_u \le y_v$ for all $x \in (0,\pi)$ (because we can obtain the other cases by replacing $x$ with $-x$ and by replacing $u$ with $-u$).

    Recall the $\Phi$ function from Equations (\ref{eq:Phi-u-1}) and (\ref{eq:Phi-u-2}). (As a reminder, $\Phi_u:(0,\infty) \times \RR \af \RR$ is defined so that $\Phi_u(x,y)$ always agrees with the argument of $F_u(x+iy)$ modulo $2\pi$.) Note that for each $u$, we have that $\Phi_u$ is a $\mathcal{C}^\infty$ function of $(x,y)$ (in fact, it is a real analytic function).

    If we plug in $v$ for $u$ in Equation (\ref{eq:Phi-u-1}), we obtain \begin{equation*}
        \Phi_v(x,y) = x - \arccot \frac{1-y}{x}
    \end{equation*} for all $(x,y) \in (0,\infty)\times\RR$. Furthermore, we note that by the definition of $y_v$, we have that $\Phi_v$ vanishes along $\gamma_v$ (see Equation (\ref{eq:Phi-x-y})). 

    Our general proof strategy will be to show that $\Phi_v$ is nondecreasing along $\gamma_u$. Combined with the fact that $\Phi_v(x,y)$ is a strictly decreasing function of $y$, we can conclude that $y_u \le y_v$ for all $x \in (0,\pi)$, as desired.

    We will have to take the derivative $\frac{\dd}{\dd x}(\Phi_v(x,y_u(x)))$ and show that it is nonnegative.
    
    Note that $\frac{\partial}{\partial x}\Phi_v(x,y) = 1 - \frac{1-y}{x^2+(1-y)^2}$ and $\frac{\partial}{\partial y}\Phi_v(x,y) = \frac{-x}{x^2+(1-y)^2}$. Now, for $x \in D_u$, we compute \begin{equation*}
        \begin{split}
            \frac{\dd}{\dd x} (\Phi_v(x,y_u(x))) &= \left(\frac{\partial}{\partial x} \Phi_v \right)(x,y_u) + y_u' \cdot \left(\frac{\partial}{\partial y} \Phi_v \right)(x,y_u) \\
            &= 1 - \frac{1-y_u}{x^2 + (1-y_u)^2} + y_u' \cdot \frac{-x}{x^2+(1-y_u)^2}
        \end{split}
    \end{equation*} and since $\frac{-x}{x^2+(1-y_u)^2}$ is negative, we may use Equation (\ref{eq:u-diff}) to write \begin{equation*}
        \begin{split}
            \frac{\dd}{\dd x} (\Phi_v(x,y_u(x))) &\ge 1 - \frac{1-y_u}{x^2 + (1-y_u)^2} + \frac{-y_u + (x^2 + y_u^2)}{x} \cdot \frac{-x}{x^2+(1-y_u)^2} \\
            &= 1 - \frac{1-y_u}{x^2 + (1-y_u)^2} + \frac{y_u - x^2 - y_u^2}{x^2+(1-y_u)^2} \\
            &= 1 + \frac{-1+2y_u - y_u^2 - x^2}{x^2+(1-y_u)^2} \\
            &= 0.
        \end{split}
    \end{equation*} We conclude that when $x \in D_u$, the function $\Phi_v(x,y_u(x))$ is a nondecreasing function of $x$. Since \\ $\lim_{x \searrow 0} y_u(x) = 0$, we obtain $\lim_{x \searrow 0} \Phi_v(x,y_u(x)) = 0$, from which it follows that \[\Phi_v(x,y_u(x)) \ge 0\] for all $x \in D_u$.

    As observed above, we have $\Phi_v(x,y_v(x)) = 0$ for all $x \in (0,\pi)$. So \[\Phi_v(x,y_u(x)) \ge \Phi_v(x,y_v(x))\] for all $x \in (0,\pi)$. Since $\Phi_v(x,y)$ is a strictly decreasing function of $y$, we must have \[y_u(x) \le y_v(x),\] as desired.
\end{proof}

%Now let us define $\tilde{F}_{\hat{u}}(x) \coloneqq F_{\hat{u}}(x+iy_{\hat{u}}(x))$. We know that $\tilde{F}_{\hat{u}}$ is $\mathcal{C}^\infty$ (in fact, real analytic) on $E_{\hat{u}}$. 
We turn our attention to Equation (\ref{eq:log-ineq}). For $x \in D_u$, we compute \begin{equation*}
    \begin{split}
        &\qquad \frac{\dd}{\dd x}\tilde{F}_u(x)\\ &= \tilde{F}_u(x) \cdot \left( \sum_{j=1}^{n+1} iu_j\left(1+iy_u'\right) - \frac{iu_j(1+iy_u')}{1+iu_j(x+iy_u)} \right) \\
        &= \tilde{F}_u(x) \cdot \sum_{j=1}^{n+1} \frac{-u_j^2(x+iy_u)(1+iy_u')}{1+iu_j(x+iy_u)} \\
        &= \tilde{F}_u(x) \cdot \left( \sum_{j=1}^{n+1} \frac{-u_j^2(x+iy_u)}{1-u_jy_u+iu_jx} + y_u' \cdot \sum_{j=1}^{n+1}  \frac{-iu_j^2(x+iy_u)}{1-u_jy_u+iu_jx}\right) \\
        %&= \tilde{F}_u(x) \cdot \left( \sum_{j=1}^{n+1} \frac{-u_j(x+iy_u)}{\frac{1}{u_j}-y_u+ix} + y_u' \cdot \sum_{j=1}^{n+1}  \frac{-iu_j(x+iy_u)}{\frac{1}{u_j}-y_u+ix}\right) \\
        %&= \tilde{F}_u(x) \cdot \left( \sum_{j=1}^{n+1} \frac{-u_j(x+iy_u)\left(\frac{1}{u_j}-y_u-ix\right)}{x^2+\left(\frac{1}{u_j}-y_u\right)^2} + y_u' \cdot \sum_{j=1}^{n+1}  \frac{-iu_j(x+iy_u)\left(\frac{1}{u_j}-y_u-ix\right)}{x^2+\left(\frac{1}{u_j}-y_u\right)^2}\right) \\
        %&= \tilde{F}_u(x) \cdot \left( \sum_{j=1}^{n+1} \frac{(x+iy_u)(-1+u_jy_u+iu_jx)}{x^2+\left(\frac{1}{u_j}-y_u\right)^2} + y_u' \cdot \sum_{j=1}^{n+1}  \frac{i(x+iy_u)(-1+u_jy_u+iu_jx)}{x^2+\left(\frac{1}{u_j}-y_u\right)^2}\right) \\
        &= \tilde{F}_u(x) \cdot \left( \sum_{j=1}^{n+1} \frac{-x+i(-y_u+u_j\cdot (x^2+y_u^2))}{x^2+\left(\frac{1}{u_j}-y_u\right)^2} + y_u' \cdot \sum_{j=1}^{n+1}  \frac{-(-y_u+u_j\cdot (x^2+y_u^2))-ix}{x^2+\left(\frac{1}{u_j}-y_u\right)^2}\right) \\
        &= \tilde{F}_u(x) \cdot \left( \sum_{j=1}^{n+1} \frac{-x}{x^2+\left(\frac{1}{u_j}-y_u\right)^2} + y_u' \cdot \sum_{j=1}^{n+1}  \frac{-(-y_u+u_j\cdot (x^2+y_u^2))}{x^2+\left(\frac{1}{u_j}-y_u\right)^2}\right) \\
    \end{split}
\end{equation*} where the last step follows from Equation (\ref{eq:diff-y}) (namely, we cancelled the imaginary part). Using Equation (\ref{eq:diff-y}) again, we see that %\[\frac{\dd}{\dd x}\tilde{F}_u(x) = \tilde{F}_u(x) \cdot (-1) \cdot \frac{\left(\sum_{j=1}^{n+1}\frac{x}{x^2+\left(\frac{1}{u_j}-y_u\right)^2}\right)^2 + \left( \sum_{j=1}^{n+1}\frac{-y_u+ u_j \cdot (x^2+ y_u^2)}{x^2+\left(\frac{1}{u_j}-y_u\right)^2}\right)^2}{\sum_{j=1}^{n+1}\frac{x}{x^2+\left(\frac{1}{u_j}-y_u\right)^2}}\] and thus 
\begin{equation}\label{eq:diff-log-u}\frac{\dd}{\dd x}\log \tilde{F}_u(x) = \left.\frac{\dd}{\dd x}\tilde{F}_u(x) \right/ \tilde{F}_u(x) = -\frac{\left(\sum_{j=1}^{n+1}\frac{x}{x^2+\left(\frac{1}{u_j}-y_u\right)^2}\right)^2 + \left( \sum_{j=1}^{n+1}\frac{-y_u+ u_j \cdot (x^2+ y_u^2)}{x^2+\left(\frac{1}{u_j}-y_u\right)^2}\right)^2}{\sum_{j=1}^{n+1}\frac{x}{x^2+\left(\frac{1}{u_j}-y_u\right)^2}}\end{equation} holds for all $x \in D_u$.

In particular, by substituting $u = v$ we obtain \begin{equation}\label{eq:diff-log-v}\frac{\dd}{\dd x} \log \tilde{F}_v(x) = -\frac{x^2+y_v^2}{x}\end{equation} for all $x \in (0,\pi)$.

%We wish to obtain $\frac{\dd}{\dd x}\log\tilde{F}_{\hat{u}}(x) \ge \frac{\dd}{\dd x}\log\tilde{F}_{\hat{v}}(x)$ for all $x \in (0,\pi)$ (which is exactly Equation (\ref{eq:log-ineq})). We can do this by comparing Equations (\ref{eq:diff-log-u}) and (\ref{eq:diff-log-v}).

We are ready to use Equations (\ref{eq:diff-log-u}) and (\ref{eq:diff-log-v}) to prove Equation (\ref{eq:log-ineq}). \textbf{\textit{The crux is to use the Cauchy-Schwarz inequality to obtain the following curious inequalities:}}  \begin{equation}\label{eq:cauchy-1}
    \begin{split}
        \left(\sum_{j=1}^{n+1}\frac{x}{x^2+\left(\frac{1}{u_j}-y_u\right)^2}\right)^2 &= \left(\sum_{j=1}^{n+1}\frac{(x/u_j) \cdot u_j}{x^2+\left(\frac{1}{u_j}-y_u\right)^2}\right)^2\\
        &\le \left( \sum_{j=1}^{n+1}\frac{(x/u_j)^2 }{\left(x^2+\left(\frac{1}{u_j}-y_u\right)^2\right)^2} \right) \left( \sum_{j=1}^{n+1} u_j^2 \right) \\
        &= \sum_{j=1}^{n+1}\frac{(x/u_j)^2 }{\left(x^2+\left(\frac{1}{u_j}-y_u\right)^2\right)^2}
    \end{split}
\end{equation} and 
\begin{equation}\label{eq:cauchy-2}
    \begin{split}
        \left( \sum_{j=1}^{n+1}\frac{-y_u+ u_j \cdot (x^2+ y_u^2)}{x^2+\left(\frac{1}{u_j}-y_u\right)^2}\right)^2 &= \left( \sum_{j=1}^{n+1}\frac{(-y_u/u_j + x^2+ y_u^2) \cdot u_j}{x^2+\left(\frac{1}{u_j}-y_u\right)^2}\right)^2 \\
        &\le \left( \sum_{j=1}^{n+1} \frac{(-y_u/u_j + x^2+ y_u^2)^2}{\left(x^2+\left(\frac{1}{u_j}-y_u\right)^2\right)^2} \right) \left( \sum_{j=1}^{n+1} u_j^2 \right) \\
        &= \sum_{j=1}^{n+1} \frac{(-y_u/u_j + x^2+ y_u^2)^2}{\left(x^2+\left(\frac{1}{u_j}-y_u\right)^2\right)^2}.
    \end{split}
\end{equation} Thus \begin{equation*}
    \begin{split}
        & \left(\sum_{j=1}^{n+1}\frac{x}{x^2+\left(\frac{1}{u_j}-y_u\right)^2}\right)^2 + \left( \sum_{j=1}^{n+1}\frac{-y_u+ u_j \cdot (x^2+ y_u^2)}{x^2+\left(\frac{1}{u_j}-y_u\right)^2}\right)^2 \\
        &\le \sum_{j=1}^{n+1}\frac{(x/u_j)^2 }{\left(x^2+\left(\frac{1}{u_j}-y_u\right)^2\right)^2} + \sum_{j=1}^{n+1} \frac{(-y_u/u_j + x^2+ y_u^2)^2}{\left(x^2+\left(\frac{1}{u_j}-y_u\right)^2\right)^2} \\
        &= \sum_{j=1}^{n+1} \frac{(x/u_j)^2 + (-y_u/u_j + x^2+ y_u^2)^2}{\left(x^2+\left(\frac{1}{u_j}-y_u\right)^2\right)^2} \\
        %&= \sum_{j=1}^{n+1} \frac{x^2/u_j^2 + y_u^2/u_j^2 -2\frac{1}{u_j}y_u \cdot (x^2 + y_u^2) + (x^2 + y_u^2)^2}{\left(x^2+\left(\frac{1}{u_j}-y_u\right)^2\right)^2}\\
        %&= \sum_{j=1}^{n+1} \frac{(x^2 + y_u^2) \cdot (\frac{1}{u_j^2} - 2\frac{1}{u_j}y_u + y_u^2 + x^2)}{\left(x^2+\left(\frac{1}{u_j}-y_u\right)^2\right)^2}\\
        &= \sum_{j=1}^{n+1} \frac{(x^2 + y_u^2) \cdot \left(x^2+\left(\frac{1}{u_j}-y_u\right)^2\right)}{\left(x^2+\left(\frac{1}{u_j}-y_u\right)^2\right)^2}\\
        &= \sum_{j=1}^{n+1} \frac{x^2 + y_u^2}{x^2+\left(\frac{1}{u_j}-y_u\right)^2}.
    \end{split}
\end{equation*} Plugging this into Equation (\ref{eq:diff-log-u}) yields \begin{equation*}
    \begin{split}
        \frac{\dd}{\dd x}\log \tilde{F}_u(x) &\ge -\frac{\sum_{j=1}^{n+1}\frac{x^2 + y_u^2}{x^2+\left(\frac{1}{u_j}-y_u\right)^2}}{\sum_{j=1}^{n+1} \frac{x}{x^2+\left(\frac{1}{u_j}-y_u\right)^2}} \\
        &= -\frac{x^2 + y_u^2}{x} \\
        &\ge -\frac{x^2 + y_v^2}{x} \\
    \end{split}
\end{equation*} for all $x \in (0,\pi)$ because $-y_v \le y_u \le y_v$ holds for all $x \in (0,\pi)$ (from Lemma (\ref{lem:vuv-bound})). Combining this with Equation (\ref{eq:diff-log-v}) yields \[\frac{\dd}{\dd x}\log \tilde{F}_u(x) \ge \frac{\dd}{\dd x}\log \tilde{F}_v(x)\] which is exactly Equation (\ref{eq:log-ineq}), as desired. This proves Theorem \ref{thm:pointwise}.

\section{A final question}
We showed that $H_\text{facet}$ is the smallest central slice, but only up to a factor of $1-o(1)$.

\begin{question}
    Can we remove the factor of $1-o(1)$, so that we can say that $\vol_{n-1}(H_\mathrm{facet} \cap \Delta_n)$ is exactly the minimum value of a central section?
\end{question}

\section{Acknowledgements}
I thank Tomasz Tkocz (Carnegie Mellon University) for introducing me to this problem, and for many stimulating discussions.

I was supported by the ARCS Foundation Scholar Award.

I relied on Wolfram Mathematica Student Edition to test various hypotheses, and also to create some of the figures in this paper.

\bibliographystyle{alphaurl}
\bibliography{minimizer}

\begin{thebibliography}{MTTT24}

\bibitem[Bal86]{ball-cube}
Keith Ball.
\newblock Cube slicing in $\mathbb{R}^n$.
\newblock {\em Proceedings of the American Mathematical Society},
  97(3):465--473, 1986.
\newblock \href {https://doi.org/10.2307/2046239} {\path{doi:10.2307/2046239}}.

\bibitem[Bal88]{ball-busemann-petty}
Keith Ball.
\newblock Some remarks on the geometry of convex sets.
\newblock In Joram Lindenstrauss and Vitali~D. Milman, editors, {\em Geometric
  Aspects of Functional Analysis}, volume 1317 of {\em Lecture Notes in
  Mathematics}, pages 224--231, Berlin, Heidelberg, 1988. Springer.
\newblock \href {https://doi.org/10.1007/BFb0081743}
  {\path{doi:10.1007/BFb0081743}}.

\bibitem[Brz13]{brzezinski}
Patryk Brzezinski.
\newblock Volume estimates for sections of certain convex bodies.
\newblock {\em Mathematische Nachrichten}, 286(17-18):1726--1743, 2013.
\newblock \href {https://doi.org/10.1002/mana.201200119}
  {\path{doi:10.1002/mana.201200119}}.

\bibitem[Dir17]{dirksen}
Hauke Dirksen.
\newblock Sections of the regular simplex – volume formulas and estimates.
\newblock {\em Mathematische Nachrichten}, 290(16):2567--2584, 2017.
\newblock \href {https://doi.org/10.1002/mana.201600109}
  {\path{doi:10.1002/mana.201600109}}.

\bibitem[Fil92]{filliman}
Paul Filliman.
\newblock The volume of duals and sections of polytopes.
\newblock {\em Mathematika}, 39(1):67--80, 1992.
\newblock \href {https://doi.org/10.1112/S0025579300006860}
  {\path{doi:10.1112/S0025579300006860}}.

\bibitem[Fra99]{fradelizi}
Matthieu Fradelizi.
\newblock Hyperplane sections of convex bodies in isotropic position.
\newblock {\em Beiträge zur Algebra und Geometrie}, 40(1):163--183, 1999.
\newblock URL: \url{https://www.emis.de/journals/BAG/vol.40/no.1/13.html}.

\bibitem[Had72]{hadwiger}
Hugo Hadwiger.
\newblock Gitterperiodische punktmengen und isoperimetrie.
\newblock {\em Monatshefte für Mathematik}, 76(5):410--418, 1972.
\newblock \href {https://doi.org/10.1007/BF01297304}
  {\path{doi:10.1007/BF01297304}}.

\bibitem[Hen79]{hensley}
Douglas Hensley.
\newblock Slicing the cube in $\mathbb{R}^n$ and probability (bounds for the
  measure of a central cube slice in $\mathbb{R}^n$ by probability methods).
\newblock {\em Proceedings of the American Mathematical Society},
  73(1):95--100, 1979.
\newblock \href {https://doi.org/10.2307/2042889} {\path{doi:10.2307/2042889}}.

\bibitem[KRZ04]{fourier}
Alexander Koldobsky, Dmitry Ryabogin, and Artem Zvavitch.
\newblock Fourier analytic methods in the study of projections and sections of
  convex bodies.
\newblock In Luca Brandolini, Leonardo Colzani, Giancarlo Travaglini, and Alex
  Iosevich, editors, {\em Fourier Analysis and Convexity}, Applied and
  Numerical Harmonic Analysis, pages 119--130. Birkhäuser, Boston, MA, 2004.
\newblock \href {https://doi.org/10.1007/978-0-8176-8172-2_6}
  {\path{doi:10.1007/978-0-8176-8172-2_6}}.

\bibitem[Kö21]{konig}
Hermann König.
\newblock Non-central sections of the simplex, the cross-polytope and the cube.
\newblock {\em Advances in Mathematics}, 376, 2021.
\newblock Article 107458.
\newblock \href {https://doi.org/10.1016/j.aim.2020.107458}
  {\path{doi:10.1016/j.aim.2020.107458}}.

\bibitem[LT20]{tkocz-cross}
Ruoyuan Liu and Tomasz Tkocz.
\newblock A note on the extremal non-central sections of the cross-polytope.
\newblock {\em Advances in Applied Mathematics}, 118, 2020.
\newblock Article 102031.
\newblock \href {https://doi.org/10.1016/j.aam.2020.102031}
  {\path{doi:10.1016/j.aam.2020.102031}}.

\bibitem[MSZZ13]{moody}
James Moody, Corey Stone, David Zach, and Artem Zvavitch.
\newblock A remark on the extremal non-central sections of the unit cube.
\newblock In Monika Ludwig, Vitali~D. Milman, Vladimir Pestov, and Nicole
  Tomczak-Jaegermann, editors, {\em Asymptotic Geometric Analysis}, volume~68
  of {\em Fields Institute Communications}, pages 211--228, New York, NY, 2013.
  Springer.
\newblock \href {https://doi.org/10.1007/978-1-4614-6406-8_9}
  {\path{doi:10.1007/978-1-4614-6406-8_9}}.

\bibitem[MTTT24]{myroshnychenko}
Serhii Myroshnychenko, Colin Tang, Kateryna Tatarko, and Tomasz Tkocz.
\newblock Stability of simplex slicing and lipschitzness of sections, 2024.
\newblock URL: \url{https://arxiv.org/abs/2403.11994}, \href
  {https://arxiv.org/abs/2403.11994} {\path{arXiv:2403.11994}}, \href
  {https://doi.org/10.48550/arXiv.2403.11994}
  {\path{doi:10.48550/arXiv.2403.11994}}.

\bibitem[NT23]{tkocz-survey}
Piotr Nayar and Tomasz Tkocz.
\newblock Extremal sections and projections of certain convex bodies: a survey,
  2023.
\newblock URL: \url{https://arxiv.org/abs/2210.00885}, \href
  {https://arxiv.org/abs/2210.00885} {\path{arXiv:2210.00885}}, \href
  {https://doi.org/10.48550/arXiv.2210.00885}
  {\path{doi:10.48550/arXiv.2210.00885}}.

\bibitem[Web96]{webb}
Simon Webb.
\newblock Central slices of the regular simplex.
\newblock {\em Geometriae Dedicata}, 61(1):19--28, 1996.
\newblock \href {https://doi.org/10.1007/BF00149416}
  {\path{doi:10.1007/BF00149416}}.

\end{thebibliography}

\end{document}